\documentclass[11pt]{article}

\usepackage{amsfonts, amssymb}
\usepackage{amsmath}
\usepackage{color}
\usepackage{indentfirst}
\usepackage{verbatim}   
\usepackage{mathrsfs}   
\usepackage{yhmath}     

\newtheorem{theorem}{Theorem}[section]
\newtheorem{proposition}[theorem]{Proposition}
\newtheorem{lemma}[theorem]{Lemma}
\newtheorem{corollary}[theorem]{Corollary}
\newtheorem{prop-def}{Proposition-Definition}[section]

\newtheorem{defi}[theorem]{Definition}

\newenvironment{proof}{\trivlist \item[\hskip \labelsep{\it Proof.}]}{
 \endtrivlist}

\begin{document}

\title{Bijection between Conjugacy Classes and Irreducible \\Representations of Finite Inverse Semigroups
\footnote{Project partially supported by national NSF of China (No 11171202).}
}
\date{}
\author{Zhenheng Li, Zhuo Li}
\maketitle

\vspace{ -1.2cm}

\vskip 7mm

\begin{abstract}
In this paper we show that the irreducible representations of a finite inverse semigroup $S$ over an algebraically closed field $F$ are in bijection with the conjugacy classes of $S$ if the characteristic of $F$ is zero or a prime number that does not divide the order of any maximal subgroup of $S$.

\vspace{ 0.3cm}
\noindent
{\bf Keywords:} Conjugacy, Irreducible representation, Finite inverse semigroup.

\vspace{ 0.3cm}
\noindent {\bf Mathematics Subject Classification 2010:}
20M18, 20M32.

\end{abstract}

\def\J {{\cal J}}
\def\a {\alpha}
\def\b {\beta}
\def\e {\varepsilon}
\def\s {\sigma}
\def\v {\vec}
\def\js {(\cal J, \sigma)}
\def\R {{\bf R}_n}

\def\F {{\cal F}}
\def\ua {\underline a}
\def\ut {\underline t}
\def\e {\varepsilon}

\def\n { {\bf n} }
\def\r { {\bf r} }
\def\ek {\eta_{K}}
\def\eJ {\eta_{J}}
\newcommand{\be }{ \begin {}\\ \end{} }

\def\l {\langle}
\def\ra {\rangle}

\baselineskip 18pt

\section{Introduction}
A useful result in group representation theory is that there is a bijection between the conjugacy classes of a finite group $G$ and the irreducible representations of $G$ over an algebraically closed field $F$ with characteristic not a factor of the order of $|G|$. Is this the same for the case of semigroup representation theory?

There are different conjugacy relations in semigroup theory. The following two are natural generalizations of the usual group conjugacy. Let $S$ be a monoid with unit group $G$. Two elements $a, b\in S$ are $G$-conjugate, denoted by $a \sim_G b$, if there is $g\in G$ such that $b =  g a g^{-1}$. Let $S$ be a semigroup. Then elements $a, b\in S$ are called {\it primarily $S$-conjugate}, denoted by $a \sim_p b$, if there are $x, y\in S$ for which $a=xy$ and $b=yx$. This relation is reflexive and symmetric, but not transitive. Its transitive closure is referred to as $S$-conjugacy and will be denoted by $\sim$. If $S$ is a group, then $\sim_G$ and $\sim$ coincide with the usual group conjugacy. If $S$ is a monoid and $a \sim_G b$ then  $a\sim b$, in other words, $\sim_G$ is finer than $\sim.$

Which conjugacy will lead to an affirmative answer to the question above? Not the $\sim_G$-conjugacy. Indeed, there are usually more $\sim_G$-conjugacy classes in a monoid than its irreducible representations over $F$. For example, the rook monoid of size 3 has 10 $\sim_G$-conjugacy classes but 7 irreducible representations (see \cite{LLC, LS2} for more details). The $S$-conjugacy will do.

There are recently a great deal of developments about $S$-conjugacy; we indicate briefly a few of them. Kudryavtseva investigated this conjugacy in regular epigroups with many elegant results \cite{Ku}, and is one of the main references for this paper. Lallement studied $S$-conjugacy for free semigroups \cite{La}. Ganyushkin, Kormysheva and Mazorchuk described $S$-conjugacy in the symmetric inverse monoid $R_n$, showing that two elements are $S$-conjugate if and only if they have the same stable rank and the restrictions to their stable images have the same cycle type \cite{GK, GM}. Kudryavtseva and Mazorchuk then characterized $S$-conjugacy in Brauer-type semigroups and semigroups of square matrices \cite{KM1, KM2}.

After the first wave by Clifford \cite{C, C2}, Lallement and Petrich \cite{LP}, Munn \cite{M0, M1, M2}, Ponizovskii \cite{P}, and Preston \cite{Pr, Pr2, Pr3}, there were many new developments in representation theory of finite semigroups. Rhodes and Zalcstein \cite{RZ} obtained explicit constructions of the irreducible representations. Bidigare {\it et al} \cite{BHR} and Brown \cite{B, B2} found applications of semigroup representations to random walks, and established connections to Solomon's descent algebra. Putcha \cite{PU3, PU5} studied the representation theory of arbitrary finite monoids and determined all the irreducible characters of full transformation semigroups.
Using Harish-Chandra's theory of cuspidal representations of finite groups of Lie type, Okn\'{\i}nski and Putcha \cite{OP} showed that every complex representation of a finite monoid of Lie type is completely reducible.

Solomon \cite{LS2} reformulated the Munn theory and determined the irreducible representations of $R_n$, with useful applications, by introducing central idempotents in the monoid algebra $FR_n$ where $F$ is a field of characteristic 0. Quite recently in \cite{LLC, LLC2}, this theory has been generalized to any Renner monoid. It turns out that the irreducible representations of the Renner monoid are completely determined by those of the parabolic subgroups of the Weyl group, and the number of inequivalent irreducible representations of the Renner monoid is the same as the number of $S$-conjugacy classes of the monoid.

Steinberg \cite{S1, S2} went even further, generalizing Solomon's approach to any finite inverse semigroup $S$.
Using the M\"obius function on $S$, Steinberg found the decomposition of the monoid algebra into a direct sum of matrix algebras over group rings and obtained the character formula for multiplicities. His formula is versatile in that he gave applications to decomposing tensor powers and exterior products of rook matrix representations in a more general setting.

The representation story is long, but one part is missing: whether there is a bijection between the irreducible representations of a finite inverse semigroup $S$ over an algebraically closed field and $S$-conjugacy classes? The work of Kudryavtseva \cite{Ku} on $S$-conjugacy in regular epigroups and that of Steinberg \cite{S1, S2} on the representations of finite inverse semigroups lead the author to give an affirmative answer.

The organization of the paper is as follows. Section 2 provides necessary facts and background information on $S$-conjugacy in a semigroup. Section 3 is the main section and establishes a bijection between $S$-conjugacy classes of a finite inverse semigroup and the irreducible representations of $S$ via the usual group conjugacy classes of all the maximal subgroups of $S$.

\section{Preliminaries}

Let $S$ be a semigroup and $a\in S$. Denote by $D_a$ and $H_a$ the $\mathcal D$-class and $\mathcal H$-class at $a$ in $S$, respectively (see standard textbooks \cite{CP} or \cite{Ho} for Green relations). An element $a\in S$ is a group-bound element if there exists a positive integer $k$ such that $a^k$ lies in a subgroup of $S$. If every element of $S$ is group-bound, we call $S$ an epigroup, which is also named as a group-bound semigroup or strongly $\pi$-regular semigroup in the literature. Every finite semigroup is an epigroup; so is the full matrix monoid consisting of all square matrices over a field. Let $a\in S$ be group-bound such that $H_{a^k}$ is a group whose identity element is denoted by $e_a$. It follows from Lemma 1 of \cite{Ku} that the identity element $e_a$ is well-defined. We call the element $e_a$ the {\it idempotent induced from} $a$ and the element $a e_a$ the {\it invertible part} of $a$.

An element $a$ of a semigroup $S$ is called regular if there exists $b$ in $S$ such that $aba = b.$ The semigroup $S$ is called regular if all its elements are regular. The elements $a, b\in S$ are referred to as mutually inverse if $a = aba$ and $b=bab.$
The following results taken from \cite{Ku} are key in our discussion.

\begin{theorem}\label{SConjugacyResults}
Let $S$ be a semigroup and $a, b\in S$.

{\rm (a)} The invertible part $a e_a$ of $a$ is a group element and $\mathcal H$-related to $e_a$, and $a e_a = e_a a$.

{\rm (b)} If $a$ and $b$ are group elements and $a \sim b$, then $a \mathcal D b$ and $a \sim_p b$.

{\rm (c)} If $a$ and $b$ are group elements with $a \mathcal H b$ and $a \sim_p b$, then there exists $h\in H_{e_a}$ such that $a=hbh^{-1}$.

{\rm (d)} If $S$ is regular and $a$ is group-bound, then  $a \sim a e_a$.

{\rm (e)} If $S$ is a regular epigroup, then $a \sim b$ if and only if $ae_a \sim be_b$ if and only if there exists a pair of mutually inverse elements $u, v \in S$ such that
$a e_a = u (b e_b) v$ and $b e_b = v (a e_a) u.$
\end{theorem}

\subsection{Inverse semigroups}
A semigroup $S$ is an inverse semigroup if each element $a\in S$ has a unique inverse $a^{-1}\in S$. It follows from the Vagner-Preston theorem that an inverse semigroup can be embedded faithfully into a symmetric inverse monoid $I_X$ on a set $X$ (see Theorem 5.1.7 of \cite{Ho}), where $I_X$ consists of all partial permutations of $X$. A partial permutation of $X$ is a bijection $a: I \rightarrow J$ with $I, J \subseteq X.$ The sets $I$ and $J$ are called the domain and range of $a$, respectively.

The binary product of two partial permutations in $I_X$ is the usual composition of partial functions. The identity permutation, denoted by $1$, is the identity element of $I_X$, and the empty partial permutation, denoted by 0, is the zero element. The unit group of $I_X$ consists of full permutations of $X$, this  group is isomorphic to the symmetric group on $X$.

If $X=\{1,2,\ldots,n\}$, we write $R_n$ for $I_X$. Notice that $R_n$ is called the rook monoid of size $n$ in combinatorics. Indeed, $R_n$ can be identified with the set of zero-one matrices which have at most one entry equal to 1 in each row and column. For instance with
${\bf n}=\{1, 2, 3, 4\}$, the following partial permutation
\[
    a =    \begin{pmatrix}
                    1 & 2 & 3 & 4 \\
                    - & 1 & 2 & 3 \\
                \end{pmatrix}
\]
corresponds to the matrix
$$
                        \begin{pmatrix}
                        0&1&0&0\\
                        0&0&1&0\\
                        0&0&0&1\\
                        0&0&0&0\\
                        \end{pmatrix}
$$
which has a one in position $i, j$ if $a(j) = i$ and zero otherwise. The domain of this partial permutation is $\{ 2, 3, 4\}$ (column indices of 1's in the matrix) and range $\{1, 2, 3\}$ (row indices of the 1's).

We gather some basic properties and notation from Lawson \cite{Law} and Steinberg \cite{S2}. If $\s\in R_n$, denote by dom($\s$) the domain of $\s\in R_n$ and ran($\s$) the range. Then
\[
\begin{aligned}
   \s^{-1}\s  &= 1_{\mbox{dom}(\s)}    \\
    \s\s^{-1} &= 1_{\mbox{ran}(\s)},
\end{aligned}
\]
where $\s$ acts on the left of $\n$. Embedding an inverse semigroup $S$ into $R_n$ via rook matrices, we can regard $a^{-1}a$ the domain of $a$ and $aa^{-1}$ the range of $a$. This allows us to write
\[
\begin{aligned}
   a^{-1}a  &= \mbox{dom}(a)    \\
    aa^{-1} &= \mbox{ran}(a),
\end{aligned}
\]
where we identify a partial identity on a subset with the subset itself. Furthermore, this gives us the freedom to think of $a$ as a bijection from dom$(a)$ to ran$(a)$.

Denote by $E(S)$ the set of idempotents of $S$, which are crucial in determining the structure and representations of $S$. Define for every $e\in E(S)$
\[
    G(e) = \{a \in S \mid \text{dom}(a) = e = \text{ran}(a)\}.
\]
Then $G(e)$ is a permutation group with the identity element $e$, and it is the same as the $\mathcal H$-class $H_e$, the maximal subgroup of $S$ at $e$.

If $e, f\in S$ are $\mathcal D$-related, it follows from Proposition 2.3.5 of \cite{Ho} that there exists an element $t\in S$ such that dom$(t) = f$ and ran$(t) = e.$  Fix such $t$. Then the following mapping given by
\begin{equation}\label{sigmat}
    \s_t: \quad a \mapsto tat^{-1} \quad \text{for all } a\in G(f)
\end{equation}
is a group isomorphism of $G(f)$ onto $G(e)$.

\section{$S$-conjugacy and representations}
From now on, we assume that $S$ is a finite inverse semigroup, though some of the following arguments are valid for infinite inverse semigroups. Moreover, the $S$-conjugacy class of $a$ in $S$ will be denoted by $[a]$.

Let $\Lambda$ be a set of idempotents of $S$ such that the $\mathcal D$-classes $\{D_e\mid e\in \Lambda\}$ are a partition of $S$. Fix such $\Lambda$. Then
\[
    S = \bigsqcup_{e\in \Lambda} D_e.
\]

\begin{defi}
Let $a$ be an element of $S$. If its invertible part $ae_a$ lies in $D_{e}$ for some $e \in \Lambda$, then $e$ is called the subrank of $a$.
\end{defi}

\begin{lemma}\label{conjugation} Let $a\in S$.

{\rm (a)} All elements in $[a]$ have the same subrank.

{\rm (b)} If $a$ has subrank $e \in \Lambda$ and $e_a$ is the idempotent induced from $a$, then there exists $t\in S$ such that $t^{-1}t = e_a$ and $tt^{-1} = e$. Furthermore, $tat^{-1} \in [a]$.
\end{lemma}
\begin{proof} The proof of (a) is straightforward by Theorem \ref{SConjugacyResults} (b) and (e). As for (b), by Theorem \ref{SConjugacyResults} (a) we see that the invertible part $ae_a$ of $a$ is $\mathcal H$-related to the idempotent $e_a$ induced from $a$, so $ae_a\in D_{e_a}$. But $ae_a\in D_e$ by the assumption. Thus $e_a \in D_e$.
It follows from Proposition 2.3.5 of \cite{Ho} that there is $t\in S$ such that $t^{-1}t = e_a$ and $tt^{-1} = e,$ which imply that $te_at^{-1} = e$ and $t^{-1}et = e_a$. Write $b=tat^{-1}$. Then the idempotent $e_b$ induced from $b$ is equal to $e$. In addition, $be_b = (tat^{-1})(te_at^{-1}) = tae_at^{-1}$ and $ae_a = t^{-1}be_b t$. In view of Theorem \ref{SConjugacyResults} (e), we deduce that $tat^{-1}\in [a]$.  $\hfill\Box$
\end{proof}

However, $t[a]t^{-1} \ne [a]$ in general; it can even happen that $t[a]t^{-1} \nsubseteq [a]$. For example, in the rook monoid $R_3$, let $a = (32)[1]$. Then $[a] = \{(12)[3], (32)[1], (31)[2]\}$, where the notation $(ij)[k]$ means the partial permutation
\[
\begin{pmatrix}
    i & j & k \\
    j & i & - \\
\end{pmatrix}
\]
for $1\le i, j, k \le 3$ and no two of them being the same. Fix $\Lambda = \{0, e_1, e_2, 1\}$, where
$e_1 = {\rm diag}\{1, 0, 0\}$ and $e_2 = {\rm diag}\{1, 1, 0\}$. The subrank of $[a]$ is $e_2$. But $e_a=a^2 = {\rm diag}\{0, 1, 1\}$.
Take $$
t = \begin{pmatrix}
            0&1&0\\
            0&0&1\\
            0&0&0\\
    \end{pmatrix}\in R_3
$$
with dom$(t) = \{2, 3\}$ and ran$(t) = \{1, 2\}$. So $t=[321]$ and $t^{-1} = [123]$. A simple calculation yields that the element $t \big((31)[2]\big)t^{-1} = 0 \notin [a].$

\begin{lemma}\label{meetW}
Each $S$-conjugacy class $[a]$ in $S$ meets a unique maximal subgroup $G(e)$ with $e\in \Lambda$. More specifically, $e$ is the subrank of $a$.
\end{lemma}
\begin{proof}
Let $e\in \Lambda$ be the subrank of $a$. The results of Theorem \ref{SConjugacyResults} (b) and (d) imply that $ae_a \in [a]\cap G(e_a).$ Since $e_a\in D_e$, there exists an element $t\in S$ such that dom$(t) = e_a$ and ran$(t) = e.$ The conjugation by $t$ is a group isomorphism of $G(e_a)$ onto $G(e)$. Observe that $tat^{-1} = tae_at^{-1}$. Thanks to Lemma \ref{conjugation} (b) we obtain that $t(ae_a)t^{-1}\in [a] \cap G(e)$. The uniqueness of $G(e)$ follows from the fact that all elements of $[a]$ have the same subrank. $\hfill\Box$
\end{proof}

\begin{proposition}\label{invertiblePartConjugacy} Let $a\in S$ have subrank $e\in \Lambda$. Then  $[a] \cap G(e) = \overline{\s_t(ae_a)}$, the usual group conjugacy class of $\s_t(ae_a)$ in $G(e)$, where $\s_t$ is defined as in {\rm(\ref{sigmat})}.
\end{proposition}
\begin{proof}
We first show that  $[a] \cap G(e_a) = \overline{ae}_a$, the usual group conjugacy class of $ae_a$ in $G(e_a)$.
Clearly $\overline{ae}_a \subseteq [a] \cap G(e_a).$ We now prove the reverse inclusion. Let $x\in [a] \cap G(e_a)$. Note that $x$ and $ae_a$ are group elements in $G(e_a)$, and $x$ is $\mathcal H$-related to $ae_a$. It follows from Theorem \ref{SConjugacyResults} (b) that $x \sim_p ae_a.$ By Theorem \ref{SConjugacyResults} (c) there is an element $g\in G(e_a)$ for which $x = g (ae_a) g^{-1}$. Thus, $x\in \overline{ae}_a.$

Let $b = tat^{-1}$. A similar argument to that of Lemma \ref{conjugation} (b) shows that the invertible part of $b$ is $be$, where $e$ is the subrank of $a$.
Then $[b] \cap G(e) = \overline{be}$, the usual group conjugacy class of $be$ in $G(e)$. Thus $[a] \cap G(e) = \overline{be}$, since $[b]=[a]$. On the other hand, it is a simple matter that $t (\overline{ae}_a) t^{-1} = \overline{be}.$ Therefore, $[a] \cap G(e) = \overline{\s_t(ae_a)}$.     $\hfill\Box$
\end{proof}

\begin{theorem}\label{bijectionBetweenClasses}
There is a bijection between the set of $S$-conjugacy classes of a finite inverse semigroup $S$ and the set of group conjugate classes of the maximal subgroups $G(e)$ of $S$ for $e\in\Lambda$.
\end{theorem}
\begin{proof} It follows from Proposition \ref{invertiblePartConjugacy} that the conjugate classes of $S$ that meet $G(e)$ are indexed by conjugate classes of $G(e)$ for $e\in \Lambda$. Thanks to Lemma \ref{meetW} and $S = \bigsqcup_{e\in \Lambda} D_e$, the proof is complete. $\hfill\Box$
\end{proof}

\begin{corollary}\label{numOfClasses}
The number of $S$-conjugacy classes of a finite inverse semigroup equals the sum of the numbers of group conjugate classes in $G(e)$ with $e\in\Lambda$. \hfill$\Box$
\end{corollary}

We now describe representations of a finite inverse semigroup $S$ over a field $F$. Let $n_e$ be the number of idempotents in $D_e$. It follows from Theorems 4.5 and 4.6 of \cite{S2} that
\[
    FS = \bigoplus_{e\in\Lambda} FD_e
\]
and there exists an isomorphism $\psi_e$ of $FD_e$ onto $M_{n_e}(FG(e))$. Extend $\psi_e$ to an algebra homomorphism of $FS$ onto $M_{n_e}(FG(e))$ by $\psi_e(FD(f)) = 0$ for $f\in \Lambda\setminus\{e\}$. If $a\in FS$ then
\begin{eqnarray}\label{psidefinition2}
    \psi_e(a) = \sum_{i,j = 1}^{n_e} \beta_{ij}(a) E_{ij},
\end{eqnarray}
where $\beta_{ij}(a) \in FG(e)$ and $E_{ij}$ are the matrix units for $1 \le i, j \le n_e$.

Applying $\rho$ to the matrix entries of $\psi_e(a)$, we obtain that a representation $\rho$ of $FG(e)$ gives rise to a representation $\rho^*$ of $FS$ by
\begin{eqnarray}\label{rhostardefinition}
    \rho^*(a) = \sum_{i, j = 1}^{n_e} \rho\big(\beta_{i, j}(a)\big) E_{i, j}
\end{eqnarray}
with $\mbox{deg}\rho^* = n_e \mbox{deg} \rho.$ We identify a representation of $S$ with its $F$-linear extension to a representation of the semigroup algebra $FS$ and make a similar convention for representations of $G(e)$ and $FG(e)$.

\begin{lemma} \label{irreducibleRepresentations}
Let $F$ be a field, $S$ a finite inverse semigroup, and $\widehat{FG}(e)$ a full set of inequivalent irreducible representations of $FG(e)$. If the characteristic of $F$ is $0$ or a prime number not dividing the order of any maximal subgroup of $S$, then $\{\rho^* \mid \rho\in \widehat{FG}(e), ~e\in\Lambda\}$ is a full set of inequivalent irreducible representations of $FS.$
\end{lemma}
\begin{proof}
The result follows from Lemma 2.22 of \cite{LS2} with a little modification.  $\hfill\Box$
\end{proof}

We can now state our main result describing the relationship between $S$-conjugacy classes in $S$ and representations of $S$ over an algebraically closed field.

\begin{theorem} \label{mainTheorem}
Let $F$ be an algebraically closed field and $S$ a finite inverse semigroup. If the characteristic of $F$ is $0$ or a prime number not dividing the order of any maximal subgroup of $S$, then the set of inequivalent irreducible representations of $S$ over $F$ is in bijection with the set of $S$-conjugacy classes in $S$.
\end{theorem}

\begin{proof}
 Lemma \ref{irreducibleRepresentations} shows that the irreducible
representations of $S$ are in bijection with the irreducible representations of maximal subgroups $G(e)$ where $e\in\Lambda$. On the other hand, from group representation theory, the irreducible representations of $G(e)$ are in bijection with the group conjugate classes of $G(e)$, since $F$ is algebraically closed. The desired result follows from Theorem \ref{bijectionBetweenClasses}. $\hfill\Box$
\end{proof}

\begin{corollary} \label{mainCorollary} With the assumption and notation in Theorem \ref{mainTheorem}, the number of inequivalent irreducible representations of a finite inverse semigroup $S$ over $F$ equals the number of $S$-conjugacy classes in $S$.
\end{corollary}

\vspace{3mm}
\noindent
{\bf Acknowledgments}

The author would like to thank Dr. Reginald Koo for his useful suggestions.


\vspace{1mm}

\noindent Zhenheng Li\\
Department of Mathematical Sciences \\
University of South Carolina Aiken\\
Aiken, SC 29801, USA\\
\noindent Email: zhenhengl@usca.edu\\

\noindent Zhuo Li\\
Department of Mathematics \\
Xiangtan University\\
Xiangtan, Hunan 411105, P. R. China\\
\noindent Email: zli@xtu.edu.cn\\

\end{document}